\documentclass[reqno]{amsart}

\usepackage{amsmath, amsfonts, amssymb}

\title[Right Lower-Triangular Nielsen
Automorphisms]{The Lower Central Series of Right Lower-Triangular Nielsen
Automorphism Groups}
\author{C. E. Kofinas}
\address{Department of Mathematics, University of the Aegean,
Karlovassi GR-83200, Samos, Greece}
\email{kkofinas@aegean.gr}

\newtheorem{theorem}{Theorem}
\newtheorem{lemma}{Lemma}
\newtheorem{corollary}{Corollary}
\newtheorem{proposition}{Proposition}
\newtheorem{remark}{Remark}

\begin{document}

\begin{abstract}
Let $F_n$ be a free group of rank $n\geq3$, freely generated by
$x_1,\ldots,x_n$. For $1\leq j<i\leq n$, let $d_{i,j}$ denote 
the Nielsen automorphism of $F_n$ defined by $d_{i,j}(x_i)=x_ix_j$ and
$d_{i,j}(x_k)=x_k$ for $k\neq i$, and let
$D_n=\langle d_{i,j}\mid 1\leq j<i\leq n\rangle$.
We determine the lower central series of $D_n$. 
We first prove that, for $1\leq r<i\leq n$, $d_{i,r}\in
\gamma_{i-r}(D_n)\setminus\gamma_{i-r+1}(D_n)$.
For each $i=2,\ldots,n$, this calculation leads to a filtration
$\{W_{i,c}\}_{c\geq1}$ of $U_i=\langle d_{i,1},\ldots,d_{i,i-1}\rangle$.  
For every $c\geq1$, we obtain an explicit iterated semidirect-product
decomposition of $\gamma_c(D_n)$ in terms of the subgroups $W_{i,c}$,
and prove that $U_i\cap\gamma_c(D_n)=W_{i,c}$ for
$i=2,\ldots,n$. 
The construction also gives an explicit basis for each quotient
$\gamma_c(D_n)/\gamma_{c+1}(D_n)$ in terms of basic commutators and
determines the exact lower-central depth of every such commutator. 
Consequently, $D_n$ is a Magnus group.

\vskip.120 in


\vskip.120 in

\end{abstract}

\maketitle

\section{Introduction}

Let $F_n$ be a free group of rank $n\geq3$, freely generated by
$x_1,\ldots,x_n$. For distinct $i,j\in\{1,\ldots,n\}$, let
$d_{i,j}$ and $e_{i,j}$ be the Nielsen automorphisms of $F_n$
defined by
\[
\begin{split}
d_{i,j}(x_i)&=x_ix_j, \quad
d_{i,j}(x_k)=x_k \quad\text{for }k\neq i,\\
e_{i,j}(x_i)&=x_jx_i, \quad
e_{i,j}(x_k)=x_k \quad\text{for }k\neq i.
\end{split}
\]
The lower-triangular automorphism group $A_n^+$ is generated by all
$d_{i,j}$ and $e_{i,j}$ with $1\leq j<i\leq n$.
Satoh obtained a normal form and a finite presentation for $A_n^+$
and studied related homological properties in \cite{sat2}.
Related questions concerning the lower central series of triangular
automorphisms acting trivially on the abelianization of a free group
have been studied by Satoh \cite{sat1} and Darn\'e \cite{dar}, 
in connection with the Andreadakis filtration.

We study the subgroup
\[
D_n=\langle d_{i,j}\mid 1\leq j<i\leq n\rangle
\]
of $\operatorname{Aut}(F_n)$. We call $D_n$ 
the \emph{right lower-triangular Nielsen automorphism
group}: ``right'' refers to multiplication on the right in
$d_{i,j}(x_i)=x_ix_j$, while ``lower-triangular'' refers to the
condition $j<i$. We also consider the corresponding left
lower-triangular Nielsen automorphism group
\[
D_n^{\ell}=\langle e_{i,j}\mid 1\leq j<i\leq n\rangle.
\]
The groups $D_n$ and $D_n^{\ell}$ are conjugate in
$\operatorname{Aut}(F_n)$. The group $D_n^{\ell}$ is precisely the
unitriangular automorphism group considered by Erofeev and
Roman'kov. Their structure theorem \cite[Theorem~A]{er} gives a
normal series, an iterated semidirect product decomposition, and a
unique normal form for $D_n^{\ell}$. The conjugacy between $D_n$ and
$D_n^{\ell}$ transfers these structural results to $D_n$. The
linearity problem for these groups was subsequently settled by
Roman'kov \cite{rom}.

For $i=2,\ldots,n$, let
\[
U_i=\langle d_{i,1},\ldots,d_{i,i-1}\rangle.
\]
For $2\leq m\leq n$, write
\[
D_m=\langle U_2,\ldots,U_m\rangle.
\]
We record this structural description in our notation and
give a direct proof that
\[
D_n=
U_n\rtimes
\bigl(
U_{n-1}\rtimes
(\cdots\rtimes(U_3\rtimes U_2)\cdots)
\bigr),
\]
where $U_i$ is a free group of rank $i-1$ for $i=2,\ldots,n$. 
It follows immediately that $D_n$ is poly-free. We also use this decomposition 
to prove that $D_n$ has trivial centre.

Our main purpose is to determine the lower central series of $D_n$
and the intersections of its terms with the groups $U_i$ for $i=2,\ldots,n$. 
The preceding decomposition is not an iterated almost-direct product: for
$2\leq k<i\leq n$, the conjugation action of $U_k$ on
$U_i^{\mathrm{ab}}$ is non-trivial. Consequently, the
Falk--Randell description of the lower central series of an
almost-direct product does not apply directly
\cite{falkrandell}.

This non-trivial action is reflected in the different lower-central
depths of the generators. We prove that
\[
d_{i,r}\in
\gamma_{i-r}(D_n)\setminus\gamma_{i-r+1}(D_n),
\qquad
1\leq r<i\leq n.
\]
Accordingly, we assign weight $i-r$ to the generator $d_{i,r}$.
We extend these weights recursively to commutator expressions in
the free generators of $U_i$, thus obtaining a weighted degree. 
For $i=2,\ldots,n$ and $c\geq1$, let $W_{i,c}$ be the
subgroup of $U_i$ generated by the commutator expressions of
weighted degree at least $c$. Using Hall's collection process, we prove that
$\{W_{i,c}\}_{c\geq1}$ is an $N$-series of $U_i$ and that
$W_{i,c}/W_{i,c+1}$ is free abelian with basis given by the images
of the basic commutators of weighted degree $c$.

For a semidirect product $H\rtimes K$, the general formula
\[
\gamma_c(H\rtimes K)
=
\gamma_c^K(H)\rtimes\gamma_c(K)
\qquad(c\geq1)
\]
expresses its lower central series in terms of the lower central
series of $K$ and the relative lower central filtration of $H$; see
\cite[Theorem~1.1]{guaschipereiro} and
\cite[Proposition~1.33]{darsu}. 
For the decompositions $D_i=U_i\rtimes D_{i-1}$, the principal
step is therefore to identify the relative filtration. We prove that
\[
\gamma_c^{D_{i-1}}(U_i)=W_{i,c}
\qquad
(3\leq i\leq n,\ c\geq1).
\]
This determines every term of the lower central series of $D_n$,
identifies the intersection of each such term with $U_i$ for
$i=2,\ldots,n$, and provides explicit bases for all the
lower-central factors.

Our main result is the following.

\begin{theorem}\label{mainthm}
Let $n\geq3$. The following statements hold.
\begin{enumerate}
\item For every $i\in\{3,\ldots,n\}$ and every $c\geq1$,
\[
W_{i,c}=\gamma_c^{D_{i-1}}(U_i).
\]
Moreover, for every $i\in\{2,\ldots,n\}$ and every $c\geq1$,
\[
W_{i,c}
=
U_i\cap\gamma_c(D_i)
=
U_i\cap\gamma_c(D_n).
\]

\item For every $c\geq1$,
\[
\gamma_c(D_n)=
W_{n,c}\rtimes
\bigl(
W_{n-1,c}\rtimes
(\cdots\rtimes(W_{3,c}\rtimes W_{2,c})\cdots)
\bigr).
\]

\item For every $c\geq1$,
\[
\frac{\gamma_c(D_n)}{\gamma_{c+1}(D_n)}
\cong
\bigoplus_{i=2}^{n}
\frac{W_{i,c}}{W_{i,c+1}}.
\]
Moreover, the images of the basic commutators of weighted degree
$c$ in each $U_i$, $2\leq i\leq n$, together form a finite basis of
$\gamma_c(D_n)/\gamma_{c+1}(D_n)$. Every such basic commutator has
exact lower-central depth $c$ in both $D_i$ and $D_n$.

\item The group $D_n$ is a Magnus group.
\end{enumerate}
\end{theorem}

\section{Preliminaries and notation}\label{sec-prelim}

Let $G$ be a group. For $m\geq1$, let $a_1,\ldots,a_m\in G$ and
$H_1,\ldots,H_m\leq G$. We denote by
$\langle a_1,\ldots,a_m\rangle$ and
$\langle H_1,\ldots,H_m\rangle$ the subgroups of $G$ generated by
$a_1,\ldots,a_m$ and by $H_1,\ldots,H_m$, respectively. For
$a,b\in G$, we write $a^b=b^{-1}ab$. For $H\leq G$ and $g\in G$,
we write $H^g=\{h^g\mid h\in H\}=g^{-1}Hg$.
An element $g\in G$ \emph{normalizes} $H$ if $H^g=H$. A subgroup
$K\leq G$ \emph{normalizes} $H$ if every element of $K$ normalizes
$H$. 

For $a,b\in G$, their commutator is the
element $(a,b)=a^{-1}a^b=a^{-1}b^{-1}ab$.
With this convention, $(a,b)^{-1}=(b,a)$. For $a,b,c\in G$, we
shall use the standard commutator identities
\[
(ab,c)=(a,c)^b(b,c),
\qquad
(a,bc)=(a,c)(a,b)^c.
\]
Throughout the paper, commutators with more than two entries are
left-normed and are defined recursively by
\[
(a_1,a_2,\ldots,a_m)
=
((a_1,\ldots,a_{m-1}),a_m)
\qquad(m\geq3).
\] 
For $H_1,H_2\leq G$, we write
$(H_1,H_2)=\langle(h_1,h_2)\mid h_1\in H_1,\ h_2\in H_2\rangle$.

For $c\geq 1$, we write $\gamma_c(G)$ for the $c$-th term of the
lower central series of $G$, where $\gamma_1(G)=G$ and
$\gamma_{c+1}(G)=(\gamma_c(G),G)$. In particular,
$\gamma_2(G)=G'$. We write $G^{\mathrm{ab}}=G/G'$ for the
abelianization of $G$. If $g\in\gamma_c(G)\setminus\gamma_{c+1}(G)$, 
we say that $g$ has
\emph{exact lower-central depth} $c$.
An $N$-series of $G$ is a descending sequence
$G=G_1\geq G_2\geq\cdots$ of normal subgroups of $G$ such that
$(G_a,G_b)\subseteq G_{a+b}$ for all $a,b\geq 1$. 
We say that $G$ is \emph{residually nilpotent} if
$\bigcap_{c\geq 1}\gamma_c(G)=\{1\}$. We call $G$ a \emph{Magnus group} if it
is residually nilpotent and
$\gamma_c(G)/\gamma_{c+1}(G)$ is torsion-free for every $c\geq1$.
A group is called \emph{poly-free} if it admits a finite subnormal
series whose successive quotient groups are free.

Let $F_n$ be the free group of rank $n$, freely generated by
$x_1,\ldots,x_n$. The action of $\operatorname{Aut}(F_n)$ on the
abelianization of $F_n$ induces an epimorphism
\[
\pi:\operatorname{Aut}(F_n)\longrightarrow
\operatorname{Aut}(F_n/F_n').
\]
Since $F_n/F_n'$ is a free abelian group of rank $n$, we have
$\operatorname{Aut}(F_n/F_n')\cong\operatorname{GL}_n(\mathbb Z)$. 

For distinct $i,j\in\{1,\ldots,n\}$, let $d_{i,j}$ be the
automorphism of $F_n$ defined by $d_{i,j}(x_i)=x_ix_j$ and
$d_{i,j}(x_k)=x_k$ for $1\leq k\leq n$, $k\neq i$, and let
$e_{i,j}$ be the automorphism defined by $e_{i,j}(x_i)=x_jx_i$ and
$e_{i,j}(x_k)=x_k$ for $1\leq k\leq n$, $k\neq i$. For
$i=2,\ldots,n$, write
\[
U_i=\langle d_{i,1},d_{i,2},\ldots,d_{i,i-1}\rangle,
\qquad
H_i=\langle e_{i,1},e_{i,2},\ldots,e_{i,i-1}\rangle.
\]
We set
\[
D_n=\langle U_2,\ldots,U_n\rangle,
\qquad
D_n^{\ell}=\langle H_2,\ldots,H_n\rangle.
\]
For $1\leq m\leq n$, let $F_m$ be the free subgroup of $F_n$ 
freely generated by $x_1,\ldots,x_m$. For $2\leq m\leq n$, write
\[
D_m=\langle U_2,\ldots,U_m\rangle.
\]
Every element of $D_m$ preserves $F_m$ and fixes
$x_{m+1},\ldots,x_n$ pointwise. Thus restriction to $F_m$
identifies $D_m$ with the corresponding subgroup of
$\operatorname{Aut}(F_m)$.

The group $D_n^{\ell}$ is precisely the
group considered by Erofeev and Roman'kov \cite{er}. Indeed, for
$1\leq j<i\leq n$, their elementary generator $\lambda_{i,j}$,
defined by $\lambda_{i,j}(x_i)=x_jx_i$ and $\lambda_{i,j}(x_k)=x_k$ 
for $1\leq k\leq n$, $k\neq i$, is precisely
our automorphism $e_{i,j}$. Their structure theorem
\cite[Theorem~A]{er} gives a normal series, a unique normal form, and
an iterated semidirect product decomposition for $D_n^{\ell}$.

Since $\pi(d_{i,j})=\pi(e_{i,j})$ for all distinct $i,j\in\{1,\ldots,n\}$, 
the groups $D_n$ and $D_n^{\ell}$ have the same image under $\pi$. We denote
this common image by
\[
\Lambda_n=\pi(D_n)=\pi(D_n^{\ell}).
\]
With respect to the ordered basis $\overline{x}_1,\ldots,\overline{x}_n$ 
of $F_n/F_n'$, we have
\[
\pi(d_{i,j})=I+M_{j,i},
\qquad 1\leq j<i\leq n,
\]
where $M_{k,l}$ denotes the matrix with $1$ in the $(k,l)$-entry
and $0$ elsewhere. Define 
\[
\iota:\operatorname{GL}_n(\mathbb Z)
\longrightarrow\operatorname{GL}_n(\mathbb Z),
\qquad
\iota(A)=A^{-T}.
\]
This map is an automorphism, since
\[
\iota(AB)=(AB)^{-T}=A^{-T}B^{-T}
=\iota(A)\iota(B),
\qquad
\iota^2=\operatorname{id}.
\]
Via $\iota$, we identify the original group $\Lambda_n$ with the
lower unitriangular group $\iota(\Lambda_n)$. With this
identification understood, we retain the notation $\Lambda_n$ for
$\iota(\Lambda_n)$ and the notation $\pi$ for the composite
$\iota\circ\pi$. Thus, for $1\leq j<i\leq n$, we write
\[
T_{i,j}=\pi(d_{i,j})
=(I+M_{j,i})^{-T}
=I-M_{i,j}.
\]
Consequently,
\[
\Lambda_n
=
\left\langle T_{i,j}\mid 1\leq j<i\leq n\right\rangle
=
\left\langle I-M_{i,j}\mid 1\leq j<i\leq n\right\rangle.
\]
The group $\Lambda_n$ is nilpotent
of class $n-1$; when $n=3$, it is the integral Heisenberg group. 
In \cite[pp.~360--361]{mag}, Magnus obtained a presentation for $\Lambda_n$ and 
a normal form for its elements. Satoh obtained a presentation and a
normal form for the lower-triangular automorphism group
$A_n^+=\langle D_n,D_n^{\ell}\rangle$ in
\cite[Lemma~3.1 and Theorem~3.2]{sat2}. 
(In Satoh's notation, $d_{i,j}=E_{ij}$ and
$e_{i,j}=E_{i^{-1}j}^{-1}$; hence $D_n^{\ell}$ is generated by the
automorphisms $E_{i^{-1}j}$ with $1\leq j<i\leq n$.)

Let $\theta\in\operatorname{Aut}(F_n)$ be defined by
$\theta(x_r)=x_r^{-1}$ for $r=1,\ldots,n$. A direct calculation
gives $\theta d_{i,j}\theta^{-1}=e_{i,j}$ for all distinct
$i,j\in\{1,\ldots,n\}$. Hence
$\theta D_n\theta^{-1}=D_n^{\ell}$ and
$\theta U_i\theta^{-1}=H_i$ for every $i=2,\ldots,n$. Thus $D_n$
and $D_n^{\ell}$ are conjugate in $\operatorname{Aut}(F_n)$, and the
structural results of Erofeev and Roman'kov for $D_n^{\ell}$ transfer
to $D_n$. Throughout the remainder of the paper, we work with $D_n$.

\section{The structure of $D_n$}\label{sec-structure}

We first record some elementary facts about the generators of
$D_n$ that will be used throughout the paper.

\begin{lemma}\label{lem1}
\begin{enumerate}
\item For $2\leq i\leq n$, the homomorphism
\[
\varphi_i:F_{i-1}\longrightarrow U_i,
\qquad
\varphi_i(x_r)=d_{i,r}\quad (1\leq r<i),
\]
is an isomorphism. Moreover, for every $w\in F_{i-1}$,
\[
\varphi_i(w)(x_i)=x_iw,
\qquad
\varphi_i(w)(x_t)=x_t\quad(1\leq t\leq n,\ t\neq i).
\]

\item Let $1\leq l<k<i\leq n$ and $1\leq j<i$. Then
\[
(d_{i,j},d_{k,l})
=
\begin{cases}
d_{i,l}^{-1}, & j=k,\\
1, & j\neq k.
\end{cases}
\]
\end{enumerate}
\end{lemma}

\begin{proof}
\begin{enumerate}
\item Fix $i\in\{2,\ldots,n\}$. The map $\varphi_i$ is surjective
by the definition of $U_i$.
Let $w\in F_{i-1}$, and write
$w=x_{r_1}^{\varepsilon_1}\cdots x_{r_m}^{\varepsilon_m}$, where
$1\leq r_s<i$ and $\varepsilon_s\in\{\pm1\}$ for
$s=1,\ldots,m$.
Then
\[
\varphi_i(w)=
d_{i,r_1}^{\varepsilon_1}\cdots
d_{i,r_m}^{\varepsilon_m}.
\]
Since every $d_{i,r_s}^{\varepsilon_s}$ fixes
$F_{i-1}$ pointwise, it follows that
\[
\varphi_i(w)(x_i)
=
x_ix_{r_1}^{\varepsilon_1}\cdots
x_{r_m}^{\varepsilon_m}
=
x_iw,
\]
while $\varphi_i(w)(x_t)=x_t$ for every
$1\leq t\leq n$ with $t\neq i$.

If $\varphi_i(w)=1$, then $x_iw=x_i$, and hence $w=1$ in the free
group $F_{i-1}$. Thus $\varphi_i$ is injective and therefore an
isomorphism.

\item Write $\alpha=d_{i,j}$ and $\beta=d_{k,l}$, where
$1\leq l<k<i\leq n$ and $1\leq j<i$. If $j\neq k$, then
$\beta$ fixes both $x_i$ and $x_j$, while $\alpha$ fixes both $x_k$
and $x_l$. Hence $\alpha$ and $\beta$ commute. Suppose that $j=k$. Then
\[
(\alpha,\beta)(x_i)=
\alpha^{-1}\beta^{-1}\alpha\beta(x_i)=
x_ix_l^{-1}.
\]
Moreover, $(\alpha,\beta)$ fixes every $x_t$ with
$1\leq t\leq n$ and $t\neq i$.
Therefore $(\alpha,\beta)=d_{i,l}^{-1}$.
\end{enumerate}
\end{proof}

The following proposition is a special case of the
structure theorem of Erofeev and Roman'kov
\cite[Theorem~A]{er}, transferred from $D_n^{\ell}$ to $D_n$ by
conjugation. Since the argument is short, we include a direct proof
in our notation.

\begin{proposition}\label{pro1}
For $n\geq3$, the group $D_n$ has the iterated semidirect product
decomposition
\[
D_n
=
U_n\rtimes
\bigl(U_{n-1}\rtimes
(\cdots\rtimes(U_3\rtimes U_2)\cdots)\bigr),
\]
where $U_i$ is free of rank $i-1$ for $i=2,\ldots,n$. 
In particular, every element of $D_n$ has a unique expression of the form
\[
u_nu_{n-1}\cdots u_2,
\qquad
u_i\in U_i\quad(i=2,\ldots,n).
\]
\end{proposition}

\begin{proof}
For $m=3,\ldots,n$, recall that
$D_m=\langle U_2,\ldots,U_m\rangle$. Let
$1\leq j<m$ and $1\leq l<k<m$. Then $d_{m,j}$ is a free
generator of $U_m$, while $d_{k,l}$ is a generator of
$D_{m-1}$. By Lemma~\ref{lem1}(2),
\[
d_{m,j}^{\,d_{k,l}}=
d_{m,j}(d_{m,j},d_{k,l})=
\begin{cases}
d_{m,k}d_{m,l}^{-1}, & j=k,\\
d_{m,j}, & j\neq k.
\end{cases}
\]
The displayed formula shows that
\[
U_m^{d_{k,l}}\subseteq U_m.
\]
Since $d_{m,l}$ is fixed by $d_{k,l}$, the inverse conjugation is
given on the free generators of $U_m$ by
\[
d_{m,j}^{\,d_{k,l}^{-1}}
=
\begin{cases}
d_{m,k}d_{m,l}, & j=k,\\
d_{m,j}, & j\neq k.
\end{cases}
\]
Hence
\[
U_m^{d_{k,l}^{-1}}\subseteq U_m,
\]
and therefore $U_m^{d_{k,l}}=U_m$. Since the elements $d_{k,l}$,
with $1\leq l<k<m$, generate $D_{m-1}$, we have
$U_m^v=U_m$ for every $v\in D_{m-1}$. Since $D_m$ is generated by
$U_m$ and $D_{m-1}$, it follows that
$U_m\trianglelefteq D_m$.

Since $D_m$ is generated by $U_m$ and $D_{m-1}$, we have
\[
D_m=U_mD_{m-1}.
\]
Moreover,
\[
U_m\cap D_{m-1}=\{1\}.
\]
Indeed, every element of $D_{m-1}$ fixes $x_m$. On the other hand,
Lemma~\ref{lem1}(1) shows that every element of $U_m$ has the form
$\varphi_m(w)$ for some $w\in F_{m-1}$ and sends $x_m$ to $x_mw$.
Such an element fixes $x_m$ only when $w=1$, in which case it is the
identity. Therefore $D_m=U_m\rtimes D_{m-1}$.

Applying this successively for $m=n,n-1,\ldots,3$ gives
\[
D_n=
U_n\rtimes
\bigl(U_{n-1}\rtimes
(\cdots\rtimes(U_3\rtimes U_2)\cdots)\bigr).
\]
The uniqueness of the expression $u_nu_{n-1}\cdots u_2$ follows
from the uniqueness of the decomposition at each semidirect-product
stage. Finally, for every
$i=2,\ldots,n$, Lemma~\ref{lem1}(1) gives
$U_i\cong F_{i-1}$, so $U_i$ is free of rank $i-1$.
\end{proof}

\begin{remark}
\upshape
A semidirect product $N\rtimes H$ is called almost direct if the
action of $H$ on $N^{\mathrm{ab}}$ is trivial. The decomposition in
Proposition~\ref{pro1} is not an iterated almost-direct product.
Indeed, if $1\leq l<j<i\leq n$, then Lemma~\ref{lem1}(2) gives
\[
(d_{i,j},d_{j,l})=d_{i,l}^{-1}.
\]
Consequently,
\[
d_{i,j}^{\,d_{j,l}}
=
d_{i,j}d_{i,l}^{-1}.
\]
Although $d_{i,l}\in D_n'$, we have $d_{i,l}\notin U_i'$, since, 
by Lemma~\ref{lem1}(1),
$d_{i,l}$ is one of the free generators of $U_i$. Therefore, writing
$\overline{u}$ for the image of $u\in U_i$ in $U_i^{\mathrm{ab}}$ and
using additive notation, we have
\[
\overline{d_{i,j}^{\,d_{j,l}}}=
\overline{d_{i,j}}-\overline{d_{i,l}}
\neq
\overline{d_{i,j}}.
\]
Hence the action of $d_{j,l}$ on $U_i^{\mathrm{ab}}$ is non-trivial.
\end{remark}

We next record a consequence of the structural decomposition.

\begin{corollary}\label{cor-structure}
For $n\geq3$, the group $D_n$ admits a poly-free series of length
$n-1$ with successive quotients that are free groups of ranks
$1,2,\ldots,n-1$.
\end{corollary}

\begin{proof}
For $2\leq i\leq n$, set
\[
P_i=\langle U_i,U_{i+1},\ldots,U_n\rangle,
\qquad
P_{n+1}=\{1\}.
\]
Proposition~\ref{pro1} gives
\[
P_i=P_{i+1}\rtimes U_i
\qquad(2\leq i\leq n).
\] 
Hence
\[
\{1\}=P_{n+1}\lhd P_n\lhd P_{n-1}\lhd\cdots
\lhd P_2=D_n,
\]
and
\[
P_i/P_{i+1}\cong U_i\cong F_{i-1}
\qquad (2\leq i\leq n).
\]
Hence $D_n$ admits a poly-free series of length $n-1$, with
successive free quotients of ranks $1,2,\ldots,n-1$.
\end{proof}

We next determine the centre of $D_n$.

\begin{proposition}\label{prop-centre}
For $n\geq3$, the centre $Z(D_n)$ is trivial.
\end{proposition}

\begin{proof}
We first consider the case $n=3$. Write
$a=d_{3,2}$, $b=d_{3,1}$ and $c=d_{2,1}$. Then
$D_3=U_3\rtimes U_2$, where $U_3=\langle a,b\rangle$ and
$U_2=\langle c\rangle$. By Lemma~\ref{lem1}(2),
\[
(a,c)=b^{-1}
\qquad\text{and}\qquad
(b,c)=1.
\]
Equivalently,
\[
a^c=ab^{-1}
\qquad\text{and}\qquad
b^c=b.
\]
An induction on $|m|$ therefore gives
\[
a^{c^m}=ab^{-m}
\qquad(m\in\mathbb Z).
\]
Let $z\in Z(D_3)$. By Proposition~\ref{pro1}, we may write
$z=uc^m$, where $u\in U_3$ and $m\in\mathbb Z$. Since both $z$ and
$c^m$ commute with $b$, the element $u$ commutes with $b$. Since
$a,b$ form a free basis of $U_3$, the centralizer of $b$ in $U_3$
is $\langle b\rangle$. Hence $u=b^q$ for some $q\in\mathbb Z$. Since $z=b^qc^m$ commutes with $a$, we have
\[
a^{b^qc^m}=a.
\]
Moreover,
\[
a^{b^q}=b^{-q}ab^q\equiv a\pmod{U_3'}.
\]
Since $c$ normalizes $U_3$, it also normalizes $U_3'$. Therefore
\[
a=a^{b^qc^m}\equiv a^{c^m}=ab^{-m}\pmod{U_3'}.
\]
Thus $b^m\in U_3'$. Since $b$ is a free generator of $U_3$, its image
in the free abelian group $U_3/U_3'$ has infinite order. 
Hence $m=0$, and therefore $z=b^q$. Since $z$ commutes with $a$, so
does $b^q$. Since $a$ is a free generator of $U_3$, the centralizer of
$a$ in $U_3$ is $\langle a\rangle$. 
Hence
\[
b^q\in
\langle a\rangle\cap\langle b\rangle
=
\{1\},
\]
because $a,b$ form a free basis of $U_3$. Therefore $q=0$, and
hence $z=1$. Thus $Z(D_3)=\{1\}$.

We now argue by induction on $n$. Let $n\geq4$ and assume that
$Z(D_{n-1})=\{1\}$. Let $z\in Z(D_n)$. By
Proposition~\ref{pro1}, we have $D_n=U_n\rtimes D_{n-1}$. 
The image of $z$ in $D_n/U_n\cong D_{n-1}$ is central and therefore
lies in $Z(D_{n-1})$. By the induction hypothesis it is trivial, so
$z\in U_n$.
Since $z$ is central in $D_n$,
it lies in $Z(U_n)$. But $U_n$ is a free group of rank
$n-1\geq2$, so its centre is trivial. Therefore $z=1$.
\end{proof}

\section{Lower-central depths}
\label{sec-depth}

We begin by determining the exact lower-central depth of the
generators of $D_n$.

\begin{proposition}\label{prop-depth}
For $n\geq i>j\geq1$, the generator $d_{i,j}$ satisfies
\[
d_{i,j}
=
(d_{j+1,j},d_{j+2,j+1},\ldots,d_{i,i-1}),
\]
with the convention that when $i-j=1$, the displayed identity means
simply $d_{i,j}=d_{j+1,j}$. In particular,
\[
d_{i,j}\in
\gamma_{i-j}(D_n)\setminus\gamma_{i-j+1}(D_n).
\]
\end{proposition}

\begin{proof}
We argue by induction on $i-j$. If $i-j=1$, the required identity is
immediate. Assume that $i-j>1$. The induction hypothesis gives
\[
d_{i-1,j}
=
(d_{j+1,j},d_{j+2,j+1},\ldots,d_{i-1,i-2}).
\]
By Lemma~\ref{lem1}(2), $(d_{i,i-1},d_{i-1,j})=d_{i,j}^{-1}$ 
and therefore $(d_{i-1,j},d_{i,i-1})=d_{i,j}$.
Combining this with the induction hypothesis, we obtain
\[
d_{i,j}
=
(d_{j+1,j},d_{j+2,j+1},\ldots,
d_{i-1,i-2},d_{i,i-1}).
\]
The right-hand side is a commutator of length $i-j$.
Hence $d_{i,j}\in\gamma_{i-j}(D_n)$.

It remains to prove that
$d_{i,j}\notin\gamma_{i-j+1}(D_n)$. For $1\leq l<k\leq n$, recall that, 
under the identification fixed above,
\[
T_{k,l}=\pi(d_{k,l})=I-M_{k,l}\in\Lambda_n.
\]
Since $M_{k,l}^2=0$, we have
$I-M_{k,l}=(I+M_{k,l})^{-1}$.
Hence Magnus's description of the lower central series of the
unitriangular group implies that, for every $s\geq1$,
\[
\gamma_s(\Lambda_n)
=
\langle T_{k,l}\mid
1\leq l<k\leq n,\ k-l\geq s\rangle.
\]
Moreover, the images of the elements $T_{k,l}$, with
$1\leq l<k\leq n$ and $k-l=s$, form a basis of the free abelian
group $\gamma_s(\Lambda_n)/\gamma_{s+1}(\Lambda_n)$;
see \cite[p.~361]{mag}. Consequently,
\[
T_{i,j}\in
\gamma_{i-j}(\Lambda_n)
\setminus
\gamma_{i-j+1}(\Lambda_n).
\]
Since homomorphisms preserve lower-central terms,
$\pi(\gamma_s(D_n))\subseteq\gamma_s(\Lambda_n)$ for
$s\geq1$.
If $d_{i,j}\in\gamma_{i-j+1}(D_n)$, then
$T_{i,j}
=
\pi(d_{i,j})
\in
\gamma_{i-j+1}(\Lambda_n)$,
which is a contradiction. Therefore
$d_{i,j}\notin\gamma_{i-j+1}(D_n)$.
\end{proof}

Thus the generators of each $U_i$ in the semidirect product
decomposition occur at different depths in the lower central series
of $D_n$. These depths determine the weights used below.

\begin{remark}
\upshape
In \cite[Remark~1]{rom}, it is stated that, for $D_4^\ell$, the
subgroup generated by
$e_{3,1},e_{3,2},e_{4,1},e_{4,2},e_{4,3}$ coincides with the
derived subgroup $(D_4^\ell)'$. Proposition~\ref{prop-depth} shows
that this identification does not hold. Indeed, since $D_4^\ell$ is
conjugate to $D_4$, Proposition~\ref{prop-depth} gives
$e_{3,2},e_{4,3}\in
D_4^\ell\setminus\gamma_2(D_4^\ell)$. Since both elements belong
to the subgroup above, whereas
$(D_4^\ell)'=\gamma_2(D_4^\ell)$, the two subgroups cannot coincide. (The asserted identification is made only in
Remark~1 of \cite{rom}, after the proof of the preceding theorem,
and is not used in that proof.)
\end{remark}

\section{Weighted commutator filtrations}
\label{sec-weighted-general}

\subsection{Basic commutators and collection}

Let $r\geq1$, and let $F_r$ be the free group freely generated by
the ordered set $X=\{x_1,\ldots,x_r\}$. We use P.~Hall's basic commutators, 
their ordering, and the collecting process described
in \cite[Chapter~11, Sections~11.1--11.2]{hall}.

For every $N\geq1$, Hall collection terminates modulo
$\gamma_{N+1}(F_r)$. Let $c_1<\cdots<c_t$ be all the basic
commutators of Hall weight at most $N$. For every $g\in F_r$, Hall's Basis Theorem
\cite[Theorem~11.2.4]{hall} gives unique integers
$e_1,\ldots,e_t$ such that
\[
g\equiv c_1^{e_1}\cdots c_t^{e_t}
\pmod{\gamma_{N+1}(F_r)}.
\]
We call the product on the right the \emph{collected expression} of
$g$ modulo $\gamma_{N+1}(F_r)$.

\subsection{Weighted filtrations of the groups $U_i$}

For $i\in\{2,\ldots,n\}$, Lemma~\ref{lem1}(1) shows that
$U_i$ is freely generated by $d_{i,1},\ldots,d_{i,i-1}$. We order
this free basis by
\[
d_{i,i-1}<d_{i,i-2}<\cdots<d_{i,1}.
\]
Applying Proposition~\ref{prop-depth} with $n=i$, we have
\[
d_{i,r}\in
\gamma_{i-r}(D_i)\setminus\gamma_{i-r+1}(D_i),
\qquad
1\leq r<i.
\]
We call $i-r$ the assigned weight of $d_{i,r}$. This should not be
confused with Hall weight: every free generator has Hall weight one,
whereas its assigned weight here is $i-r$.

A \emph{commutator expression} in the free generators of $U_i$ is
obtained recursively from these generators by inversion and the
commutator operation. Thus the generators are commutator expressions,
and if $u$ and $v$ are commutator expressions, then so are $u^{-1}$
and $(u,v)$. The ordinary degree and the weighted degree of a 
commutator expression are defined recursively by
\[
\begin{aligned}
\ell(d_{i,r})&=1,
&
\operatorname{wt}(d_{i,r})&=i-r,\\
\ell(u^{-1})&=\ell(u),
&
\operatorname{wt}(u^{-1})&=\operatorname{wt}(u),\\
\ell((u,v))&=\ell(u)+\ell(v),
&
\operatorname{wt}((u,v))
&=\operatorname{wt}(u)+\operatorname{wt}(v).
\end{aligned}
\]
For a basic commutator, its ordinary degree is its Hall weight. Since
$1\leq\operatorname{wt}(d_{i,r})\leq i-1$,
every commutator expression $u$ satisfies
\[
\ell(u)\leq\operatorname{wt}(u)\leq(i-1)\ell(u).
\]
In particular, a basic commutator of weighted degree $c$ has Hall
weight at most $c$. Since $U_i$ has finite rank, there are only
finitely many such basic commutators.

Every commutator expression $u$ satisfies
\[
u\in\gamma_{\ell(u)}(U_i)
\qquad\text{and}\qquad
u\in\gamma_{\operatorname{wt}(u)}(D_i).
\]
The first inclusion follows from the definition of the lower central
series. The second inclusion holds for the generators by
Proposition~\ref{prop-depth} and is preserved under inversion. 
For a commutator expression $(u,v)$, it follows from 
\[
(\gamma_a(D_i),\gamma_b(D_i))\subseteq\gamma_{a+b}(D_i).
\]

Let $u$ and $v$ be basic commutators, and let
$\varepsilon,\delta\in\{\pm1\}$. When $v^\varepsilon$ and
$u^\delta$ are interchanged during Hall collection, every new
commutator is an iterated commutator in $u$ and $v$ in which both
occur; see \cite[pp.~165--167]{hall}. If such a
commutator contains $u$ exactly $p$ times and $v$ exactly $q$ times,
where $p,q\geq1$, then its weighted degree is
\[
p\operatorname{wt}(u)+q\operatorname{wt}(v)
\geq
\operatorname{wt}(u)+\operatorname{wt}(v).
\]

Fix $N\geq1$. We prove by induction on the construction of
commutator expressions that every basic commutator in the collected
expression of a commutator expression $a$ modulo
$\gamma_{N+1}(U_i)$ has weighted degree at least
$\operatorname{wt}(a)$. For a commutator expression $u$, denote its
collected expression modulo $\gamma_{N+1}(U_i)$ by $\widetilde u$.

If $a$ is a free generator, the assertion is immediate. Suppose
first that $a=u^{-1}$ and that the assertion holds for $u$. Since
\[
u\equiv\widetilde u
\pmod{\gamma_{N+1}(U_i)},
\]
we have
\[
a=u^{-1}\equiv\widetilde u^{-1}
\pmod{\gamma_{N+1}(U_i)}.
\]
By the induction hypothesis, every basic commutator occurring in
$\widetilde u$ has weighted degree at least
$\operatorname{wt}(u)=\operatorname{wt}(a)$. Inverting
$\widetilde u$ reverses the order of its factors and replaces them
by their inverses, without changing their weighted degrees. When
the resulting product is collected, every new factor is an iterated
commutator in factors already present. The preceding observation
about Hall collection therefore shows that no basic commutator of
weighted degree less than $\operatorname{wt}(a)$ is introduced.

Now suppose that $a=(u,v)$ and that the assertion holds for $u$ and
$v$. Since
\[
u\equiv\widetilde u
\qquad\text{and}\qquad
v\equiv\widetilde v
\pmod{\gamma_{N+1}(U_i)},
\]
the standard commutator identities, together with
\[
(\gamma_{N+1}(U_i),U_i)
\subseteq\gamma_{N+2}(U_i)
\subseteq\gamma_{N+1}(U_i),
\]
give
\[
(u,v)\equiv(\widetilde u,\widetilde v)
\pmod{\gamma_{N+1}(U_i)}.
\]
If either $\widetilde u$ or $\widetilde v$ is trivial, the assertion
is immediate. Otherwise, after writing their powers as products of
copies of basic commutators or their inverses, repeated application
of the standard commutator identities expresses
$(\widetilde u,\widetilde v)$ as a product of conjugates of iterated
commutators containing at least one factor from $\widetilde u$ and
at least one factor from $\widetilde v$. By the induction
hypothesis, these factors have weighted degrees at least
$\operatorname{wt}(u)$ and $\operatorname{wt}(v)$, respectively.
Hence every such iterated commutator has weighted degree at least
\[
\operatorname{wt}(u)+\operatorname{wt}(v)=\operatorname{wt}(a).
\]
Conjugation and further collection introduce only commutators in
factors already present and therefore cannot introduce a basic
commutator of smaller weighted degree. This completes the induction.

\begin{lemma}\label{weightedcollection}
Let $i\in\{2,\ldots,n\}$ and let $\rho,N\geq1$. Suppose that
$w=a_1^{e_1}\cdots a_t^{e_t}$,
where $t\geq1$ and, for $j=1,\ldots,t$,
$e_j\in\mathbb Z\setminus\{0\}$ and $a_j$ is a commutator
expression in the generators of $U_i$ of weighted degree at least
$\rho$. Then every basic commutator appearing in the collected
expression of $w$ modulo $\gamma_{N+1}(U_i)$ has weighted degree at
least $\rho$.
\end{lemma}

\begin{proof}
Write each power as a product of copies of $a_j$ or $a_j^{-1}$.
The preceding collection argument shows that each of these
commutator expressions collects into basic commutators of weighted
degree at least $\rho$. Collecting the resulting product introduces
only iterated commutators in factors already present and therefore
cannot introduce a basic commutator of weighted degree less than
$\rho$.
\end{proof}

For $i=2,\ldots,n$ and $c\geq1$, let $W_{i,c}$ be the
subgroup of $U_i$ generated by the commutator expressions of
weighted degree at least $c$. The
inclusions proved above give
\begin{equation}\label{eq-weighted-in-lcs}
W_{i,c}\subseteq U_i\cap\gamma_c(D_i).
\end{equation}

\begin{lemma}\label{weightedhall}
For every $i\in\{2,\ldots,n\}$, the following statements hold.
\begin{enumerate}
\item The sequence $\{W_{i,c}\}_{c\geq1}$ is an $N$-series of $U_i$.

\item For every $c\geq1$, the group $W_{i,c}/W_{i,c+1}$ is a
finitely generated free abelian group with basis given by the images
of the basic commutators of weighted degree exactly $c$.

\item $\bigcap_{c\geq1}W_{i,c}=\{1\}$.
\end{enumerate}
\end{lemma}

\begin{proof}
Fix $i\in\{2,\ldots,n\}$.

\begin{enumerate}
\item Let $c\geq1$. If $u$ is a commutator expression of weighted
degree at least $c$, then
\[
\operatorname{wt}((u,d_{i,r}^{\varepsilon}))=
\operatorname{wt}(u)+\operatorname{wt}(d_{i,r})
\geq c+1
\]
for every $1\leq r<i$ and $\varepsilon\in\{\pm1\}$. Hence
$u^{d_{i,r}^{\varepsilon}}=
u(u,d_{i,r}^{\varepsilon})\in W_{i,c}$.
Since the elements $d_{i,r}$ generate $U_i$, the subgroup
$W_{i,c}$ is normal in $U_i$.

Let $a,b\geq1$, and let $u$ and $v$ be commutator expressions of
weighted degrees at least $a$ and $b$, respectively. Then
$(u,v)\in W_{i,a+b}$. Since $W_{i,a+b}$ is normal in $U_i$, 
the standard commutator identities from Section~\ref{sec-prelim} give
$(W_{i,a},W_{i,b})\subseteq W_{i,a+b}$.
Since $W_{i,1}=U_i$ and $W_{i,c+1}\subseteq W_{i,c}$ for every
$c\geq1$, the sequence $\{W_{i,c}\}_{c\geq1}$ is an $N$-series of
$U_i$.

\item Let $c\geq1$. Since $\{W_{i,c}\}_{c\geq1}$ is an $N$-series,
$\gamma_{c+1}(U_i)\subseteq W_{i,c+1}$.
Moreover, $(W_{i,c},W_{i,c})\subseteq W_{i,2c}\subseteq W_{i,c+1}$,
so $W_{i,c}/W_{i,c+1}$ is abelian. 
Let $\mathcal B_{i,c}$ be the finite set of basic commutators of
weighted degree exactly $c$. Let $u$ be a commutator expression of
weighted degree at least $c$. If
$\operatorname{wt}(u)\geq c+1$, then $u\in W_{i,c+1}$. Suppose that
$\operatorname{wt}(u)=c$. By Lemma~\ref{weightedcollection},
applied with $N=c$, the collected expression of $u$ modulo
$\gamma_{c+1}(U_i)$ contains only basic commutators of weighted
degree at least $c$. Since $\gamma_{c+1}(U_i)\subseteq W_{i,c+1}$ 
and every basic commutator of weighted degree at least $c+1$ 
belongs to $W_{i,c+1}$, the image of $u$ in $W_{i,c}/W_{i,c+1}$ 
lies in the subgroup generated by the images of the elements of $\mathcal B_{i,c}$.
Hence these images generate $W_{i,c}/W_{i,c+1}$.

To prove linear independence, let $b_1<\cdots<b_t$ be distinct elements of
$\mathcal B_{i,c}$, and suppose that
$b_1^{e_1}\cdots b_t^{e_t}\in W_{i,c+1}$, where
$e_1,\ldots,e_t\in\mathbb Z$.
Since each $b_j$ has Hall weight at most $c$, the product on the left
is already in collected form modulo $\gamma_{c+1}(U_i)$. Since it
belongs to $W_{i,c+1}$, it can also be written as a product of powers
of commutator expressions of weighted degree at least $c+1$.
Lemma~\ref{weightedcollection}, applied with $\rho=c+1$ and $N=c$,
shows that its collected expression modulo $\gamma_{c+1}(U_i)$
contains only basic commutators of weighted degree at least $c+1$.
The uniqueness of collected expressions therefore gives
\[
e_1=\cdots=e_t=0.
\]
Thus the images of the elements of $\mathcal B_{i,c}$ form a finite
basis of $W_{i,c}/W_{i,c+1}$. In particular, this quotient is a
finitely generated free abelian group.

\item Let $c\geq1$, and let $u$ be a commutator expression of
weighted degree at least $c$. Since
\[
c\leq\operatorname{wt}(u)\leq(i-1)\ell(u),
\]
we have
\[
\ell(u)\geq
\left\lceil\frac{c}{i-1}\right\rceil.
\]
As observed above, $u\in\gamma_{\ell(u)}(U_i)$. Since the lower
central series is descending, it follows that
\[
u\in\gamma_{\ell(u)}(U_i)
\subseteq
\gamma_{\left\lceil c/(i-1)\right\rceil}(U_i).
\]
The commutator expressions of weighted degree at least $c$ generate
$W_{i,c}$, and therefore
\[
W_{i,c}
\subseteq
\gamma_{\left\lceil c/(i-1)\right\rceil}(U_i).
\]

Now let
\[
g\in\bigcap_{d\geq1}W_{i,d}.
\]
For every $k\geq1$, taking $d=(i-1)k$ gives
\[
g\in W_{i,(i-1)k}\subseteq\gamma_k(U_i).
\]
Hence
\[
g\in\bigcap_{k\geq1}\gamma_k(U_i).
\]
Since $U_i$ is a free group, it is residually nilpotent, and thus
\[
\bigcap_{k\geq1}\gamma_k(U_i)=\{1\}.
\]
Therefore $g=1$, and consequently
\[
\bigcap_{c\geq1}W_{i,c}=\{1\}.
\]
\end{enumerate}
\end{proof}

Let $i\in\{2,\ldots,n\}$, and let $b$ be a basic commutator in the
free generators of $U_i$ of weighted degree $c\geq1$.
By Lemma~\ref{weightedhall}(2), $b\in
W_{i,c}\setminus W_{i,c+1}$, and by
equation~\eqref{eq-weighted-in-lcs},
$b\in\gamma_c(D_i)$.
The equality $W_{i,c+1}=U_i\cap\gamma_{c+1}(D_i)$,
proved in the next section, will show that
$b\notin\gamma_{c+1}(D_i)$; see
Corollary~\ref{cor-exact-weight} below.

\section{Weighted filtrations and the lower central series of $D_n$}
\label{sec-weighted}

Let $K$ act on a group $H$. The relative lower central filtration
$\{\gamma_c^K(H)\}_{c\geq1}$ is defined by
$\gamma_1^K(H)=H$ and
\[
\gamma_{c+1}^K(H)
=
(\gamma_c^K(H),H\rtimes K)
\qquad(c\geq1);
\]
see \cite[Proposition-Definition~1.29]{darsu}. By
\cite[Lemma~1.31]{darsu}, this is the smallest $N$-series of $H$ on
which the lower central series of $K$ acts. Explicitly, an
$N$-series $\{H_c\}_{c\geq1}$ has this property if
$(\gamma_a(K),H_b)\subseteq H_{a+b}$ for all $a,b\geq1$.

The filtration defined recursively by Guaschi and Pereiro in
\cite[Theorem~1.1]{guaschipereiro} is precisely
$\{\gamma_c^K(H)\}_{c\geq1}$; see
\cite[Remark~1.32]{darsu}. Consequently,
\[
\gamma_c(H\rtimes K)
=
\gamma_c^K(H)\rtimes\gamma_c(K)
\qquad(c\geq1);
\]
see \cite[Theorem~1.1]{guaschipereiro} and
\cite[Proposition~1.33]{darsu}.

By Proposition~\ref{pro1}, for every $i\in\{3,\ldots,n\}$,
$D_i=U_i\rtimes D_{i-1}$. Thus $D_{i-1}$ acts on $U_i$ by conjugation. 
The next proposition shows that, for this action, the relative lower
central filtration coincides with the weighted filtration
$\{W_{i,c}\}_{c\geq1}$.

\begin{proposition}\label{prop-weighted-intersection}
The following statements hold.
\begin{enumerate}
\item For every $i\in\{3,\ldots,n\}$ and every $c\geq1$,
\[
W_{i,c}=\gamma_c^{D_{i-1}}(U_i)=U_i\cap\gamma_c(D_i).
\]

\item For every $c\geq1$,
\[
\gamma_c(D_n)=
W_{n,c}\rtimes
\bigl(
W_{n-1,c}\rtimes
(\cdots\rtimes(W_{3,c}\rtimes W_{2,c})\cdots)
\bigr),
\]
and, for every $i\in\{2,\ldots,n\}$,
\[
U_i\cap\gamma_c(D_n)=W_{i,c}.
\]
\end{enumerate}
\end{proposition}

\begin{proof}
\begin{enumerate}
\item Fix $i\in\{3,\ldots,n\}$ and, for $c\geq1$, write
$R_c=\gamma_c^{D_{i-1}}(U_i)$. Since $D_i=U_i\rtimes D_{i-1}$, the relative lower central
decomposition gives
\[
\gamma_c(D_i)=R_c\rtimes\gamma_c(D_{i-1}).
\]
We first show that $U_i\cap\gamma_c(D_i)=R_c$.
The inclusion $R_c\subseteq U_i\cap\gamma_c(D_i)$ is immediate. Conversely,
if $g\in U_i\cap\gamma_c(D_i)$, then $g=rv$
for some $r\in R_c$ and $v\in\gamma_c(D_{i-1})$. Since
$r,g\in U_i$, we have $v=r^{-1}g\in U_i\cap D_{i-1}$. By Proposition~\ref{pro1}, $U_i\cap D_{i-1}=\{1\}$. Thus $g=r\in R_c$, proving the reverse inclusion. 

Consequently, equation~\eqref{eq-weighted-in-lcs} gives $W_{i,c}\subseteq R_c$.
It remains to prove $R_c\subseteq W_{i,c}$. For
$v\in D_{i-1}$, let
$\alpha_v\in\operatorname{Aut}(U_i)$ be defined by
$\alpha_v(u)=vuv^{-1}$. For $1\leq l<q<i$,
$\varepsilon\in\{\pm1\}$, and $1\leq r<i$,
Lemma~\ref{lem1}(2) gives
\[
\alpha_{d_{q,l}^{\varepsilon}}(d_{i,r})
=
\begin{cases}
d_{i,q}d_{i,l}^{\varepsilon},&r=q,\\
d_{i,r},&r\neq q.
\end{cases}
\]
Since
\[
d_{i,l}\in W_{i,i-l}\subseteq W_{i,i-q+1},
\]
we obtain
\[
\alpha_{d_{q,l}^{\varepsilon}}(d_{i,r})
\equiv d_{i,r}
\pmod{W_{i,i-r+1}}.
\]

Fix $1\leq l<q<i$ and $\varepsilon\in\{\pm1\}$, and write
$\alpha=\alpha_{d_{q,l}^{\varepsilon}}$. By
Lemma~\ref{weightedhall}(1), $\{W_{i,c}\}_{c\geq1}$ is an
$N$-series of $U_i$. We show by induction on the construction of
commutator expressions that
\[
\alpha(u)\equiv u\pmod{W_{i,m+1}}
\]
whenever $u$ is a commutator expression in the free generators of
$U_i$ of weighted degree $m$. Note that every such expression
belongs to $W_{i,m}$. The assertion has already been proved for the
free generators.

Suppose first that $u=a^{-1}$, where $a$ has weighted degree $m$,
and that the assertion holds for $a$. Since
$\alpha(a)\equiv a\pmod{W_{i,m+1}}$ and $W_{i,m+1}$ is normal in $U_i$, 
we have
\[
\alpha(u)=\alpha(a)^{-1}
\equiv a^{-1}=u
\pmod{W_{i,m+1}}.
\]

Now suppose that $u=(a,b)$, where $a$ and $b$ have weighted degrees
$s$ and $t$, respectively, with $s,t\geq1$ and $s+t=m$. Thus
$a\in W_{i,s}$ and $b\in W_{i,t}$. By the induction hypothesis,
\[
\alpha(a)\equiv a\pmod{W_{i,s+1}}
\qquad\text{and}\qquad
\alpha(b)\equiv b\pmod{W_{i,t+1}}.
\]
Since $\{W_{i,c}\}_{c\geq1}$ is an $N$-series of $U_i$, we have
\[
(W_{i,s+1},W_{i,t})\subseteq W_{i,s+t+1},
\qquad
(W_{i,s},W_{i,t+1})\subseteq W_{i,s+t+1}.
\]
The standard commutator identities therefore show that the image
of $(x,y)$ modulo $W_{i,s+t+1}$, for $x\in W_{i,s}$ and
$y\in W_{i,t}$, depends only on the image of $x$ modulo
$W_{i,s+1}$ and the image of $y$ modulo $W_{i,t+1}$.
Consequently,
\[
\alpha(u)=\alpha((a,b))
\equiv(a,b)=u
\pmod{W_{i,m+1}}.
\] 
This completes the induction.

Let $u$ be a commutator expression of weighted degree $m\geq c$.
The preceding congruence gives
\[
\alpha_{d_{q,l}}(u)u^{-1}\in W_{i,m+1}\subseteq W_{i,c}.
\]
Since $u\in W_{i,m}\subseteq W_{i,c}$, it follows that
$\alpha_{d_{q,l}}(u)\in W_{i,c}$. The same argument, applied to
$\alpha_{d_{q,l}}^{-1}=\alpha_{d_{q,l}^{-1}}$, shows that
$\alpha_{d_{q,l}}^{-1}$ also maps $W_{i,c}$ into itself. 
Since the commutator expressions of weighted degree at least $c$ 
generate $W_{i,c}$, we conclude that
\[
\alpha_{d_{q,l}}(W_{i,c})=W_{i,c}.
\]
Since $\alpha_{d_{q,l}}$ preserves both $W_{i,c}$ and
$W_{i,c+1}$, it induces an automorphism
\[
\overline{\alpha}_{d_{q,l}}:
W_{i,c}/W_{i,c+1}
\longrightarrow
W_{i,c}/W_{i,c+1}.
\]
If $b$ is a basic commutator of weighted degree $c$, then
$\overline{\alpha}_{d_{q,l}}\bigl(bW_{i,c+1}\bigr)=bW_{i,c+1}$.
By Lemma~\ref{weightedhall}(2), these elements form a basis of
$W_{i,c}/W_{i,c+1}$. Hence
$\overline{\alpha}_{d_{q,l}}$ is the identity automorphism.

The elements $d_{q,l}$, where $1\leq l<q<i$, generate
$D_{i-1}$. Since each $\alpha_{d_{q,l}}$ induces the identity on
$W_{i,c}/W_{i,c+1}$, the same holds for $\alpha_v$ for every
$v\in D_{i-1}$. Thus
\[
vwv^{-1}W_{i,c+1}=wW_{i,c+1}
\]
for all $v\in D_{i-1}$ and $w\in W_{i,c}$, or equivalently,
\[
vwv^{-1}w^{-1}\in W_{i,c+1}.
\]
Replacing $v$ and $w$ by their inverses gives
\[
(v,w)=v^{-1}w^{-1}vw\in W_{i,c+1}.
\] 
Hence
\[
(D_{i-1},W_{i,c})
\subseteq
W_{i,c+1}
\qquad(c\geq1).
\]

Since $\{W_{i,c}\}_{c\geq1}$ is an $N$-series of $U_i$, the
preceding inclusion and \cite[Corollary~1.25]{darsu} give
\[
(\gamma_a(D_{i-1}),W_{i,b})
\subseteq W_{i,a+b}
\qquad(a,b\geq1).
\]
Thus the lower central series of $D_{i-1}$ acts on
$\{W_{i,c}\}_{c\geq1}$. Therefore, by
\cite[Lemma~1.31]{darsu}, $R_c\subseteq W_{i,c}$.
Hence $R_c=W_{i,c}$ for every $c\geq1$.

\item Fix $c\geq1$. For every $i\in\{3,\ldots,n\}$,
Proposition~\ref{pro1} gives
$D_i=U_i\rtimes D_{i-1}$. Hence part~(1) and the
split-extension formula give
\[
\gamma_c(D_i)
=
\gamma_c^{D_{i-1}}(U_i)\rtimes\gamma_c(D_{i-1})
=
W_{i,c}\rtimes\gamma_c(D_{i-1}).
\]

Since $D_2=U_2\cong\mathbb Z$, we have
$\gamma_1(D_2)=U_2$ and $\gamma_c(D_2)=\{1\}$ for $c\geq2$.
Moreover, the generator $d_{2,1}$ has weighted degree one, while
every commutator in $U_2$ is trivial. Hence
$W_{2,1}=U_2$ and $W_{2,c}=\{1\}$ for $c\geq2$. Therefore
$\gamma_c(D_2)=W_{2,c}$ for every $c\geq1$.
We now prove the required decomposition by induction.
For $k=3$, the preceding equality gives
\[
\gamma_c(D_3)
=
W_{3,c}\rtimes\gamma_c(D_2)
=
W_{3,c}\rtimes W_{2,c}.
\]
Now let $k\geq4$ and assume that
\[
\gamma_c(D_{k-1})=
W_{k-1,c}\rtimes
\bigl(
W_{k-2,c}\rtimes
(\cdots\rtimes(W_{3,c}\rtimes W_{2,c})\cdots)
\bigr).
\]
Applying the preceding equality with $i=k$ gives
\[
\gamma_c(D_k)=W_{k,c}\rtimes\gamma_c(D_{k-1}).
\]
Substituting the induction hypothesis, we obtain
\[
\gamma_c(D_k)=
W_{k,c}\rtimes
\bigl(
W_{k-1,c}\rtimes
(\cdots\rtimes(W_{3,c}\rtimes W_{2,c})\cdots)
\bigr).
\]
This completes the induction. Therefore,
\[
\gamma_c(D_n)=
W_{n,c}\rtimes
\bigl(
W_{n-1,c}\rtimes
(\cdots\rtimes(W_{3,c}\rtimes W_{2,c})\cdots)
\bigr).
\]

We now prove that
\[
U_i\cap\gamma_c(D_n)=W_{i,c}
\qquad(i=2,\ldots,n).
\]
The inclusion
$W_{i,c}\subseteq U_i\cap\gamma_c(D_n)$ follows immediately from
the preceding decomposition. Conversely, let
$g\in U_i\cap\gamma_c(D_n)$. By the preceding decomposition, there
exist $w_j\in W_{j,c}$, for $j=2,\ldots,n$, such that
\[
g=w_nw_{n-1}\cdots w_2.
\] 
Since $W_{j,c}\subseteq U_j$, 
this is a normal form of the type given
in Proposition~\ref{pro1}. Since $g\in U_i$, the uniqueness of the normal form in
Proposition~\ref{pro1} gives $w_i=g$ and $w_j=1$ for $j\neq i$. Thus $g=w_i\in W_{i,c}$, proving the reverse inclusion.
Hence $U_i\cap\gamma_c(D_n)=W_{i,c}$ for every $i=2,\ldots,n$.
\end{enumerate}
\end{proof}

We can now determine the exact lower-central depths of the basic
commutators.

\begin{corollary}\label{cor-exact-weight}
Let $2\leq i\leq n$, and let $b$ be a basic commutator in the
free generators $d_{i,1},\ldots,d_{i,i-1}$ of weighted degree
$c\geq1$. Then
\[
b\in\gamma_c(D_i)\setminus\gamma_{c+1}(D_i).
\]
Moreover,
\[
b\in\gamma_c(D_n)\setminus\gamma_{c+1}(D_n).
\]
Thus the weighted degree of $b$ is its exact lower-central depth in
both $D_i$ and $D_n$.
\end{corollary}

\begin{proof}
By Lemma~\ref{weightedhall}(2),
$b\in W_{i,c}\setminus W_{i,c+1}$.
If $i\geq3$, Proposition~\ref{prop-weighted-intersection}(1) gives
\[
U_i\cap\gamma_c(D_i)=W_{i,c}
\qquad\text{and}\qquad
U_i\cap\gamma_{c+1}(D_i)=W_{i,c+1}.
\]
Hence $b\in\gamma_c(D_i)\setminus\gamma_{c+1}(D_i)$.
For $i=2$, the assertion is immediate since $U_2=D_2$ is infinite
cyclic and the only basic commutator on one free generator is the
generator itself, which has weighted degree one.

Finally, Proposition~\ref{prop-weighted-intersection}(2) gives
\[
U_i\cap\gamma_c(D_n)=W_{i,c}
\qquad\text{and}\qquad
U_i\cap\gamma_{c+1}(D_n)=W_{i,c+1}.
\]
Therefore
$b\in\gamma_c(D_n)\setminus\gamma_{c+1}(D_n)$.
\end{proof}

\section{Lower-central factors and residual consequences}
\label{sec-magnus}

\begin{proposition}\label{prop-lower-central-factors}
For every $n\geq3$ and every $c\geq1$,
\[
\frac{\gamma_c(D_n)}{\gamma_{c+1}(D_n)}
\cong
\bigoplus_{i=2}^{n}
\frac{W_{i,c}}{W_{i,c+1}}.
\]
Moreover, the images of the basic commutators of weighted degree
$c$ in each $U_i$, $2\leq i\leq n$, together form a finite basis of
$\gamma_c(D_n)/\gamma_{c+1}(D_n)$. In particular, this group is
finitely generated free abelian.
\end{proposition}

\begin{proof}
Fix $n\geq3$ and $c\geq1$. By
Proposition~\ref{prop-weighted-intersection}(2), for
$2\leq i\leq n$,
\[
U_i\cap\gamma_c(D_n)=W_{i,c}
\qquad\text{and}\qquad
U_i\cap\gamma_{c+1}(D_n)=W_{i,c+1}.
\]
For every $2\leq i\leq n$, define
\[
\phi_{i,c}:
\frac{W_{i,c}}{W_{i,c+1}}
\longrightarrow
\frac{\gamma_c(D_n)}{\gamma_{c+1}(D_n)},
\qquad
wW_{i,c+1}\longmapsto
w\gamma_{c+1}(D_n).
\]
This is a well-defined homomorphism since
$W_{i,c+1}\subseteq\gamma_{c+1}(D_n)$. Moreover, since
$W_{i,c}\subseteq U_i$, we have
\[
W_{i,c}\cap\gamma_{c+1}(D_n)=
W_{i,c}\cap\bigl(U_i\cap\gamma_{c+1}(D_n)\bigr)=
W_{i,c}\cap W_{i,c+1}=
W_{i,c+1}.
\]
Therefore
\[
\ker\phi_{i,c}
=
\frac{W_{i,c}\cap\gamma_{c+1}(D_n)}
{W_{i,c+1}}
=
\{1\},
\]
and hence $\phi_{i,c}$ is injective.

Since $\gamma_c(D_n)/\gamma_{c+1}(D_n)$ is abelian, the
homomorphisms $\phi_{i,c}$ determine a homomorphism
\[
\Phi_c:
\bigoplus_{i=2}^{n}\frac{W_{i,c}}{W_{i,c+1}}
\longrightarrow
\frac{\gamma_c(D_n)}{\gamma_{c+1}(D_n)}
\]
given by
\[
\Phi_c(\overline w_2,\ldots,\overline w_n)
=
\overline{w_n\cdots w_2}.
\]
By Proposition~\ref{prop-weighted-intersection}(2), every element of
$\gamma_c(D_n)$ can be written as $w_n\cdots w_2$, where
$w_i\in W_{i,c}$ for $2\leq i\leq n$. Hence $\Phi_c$ is
surjective.

Let $w_i\in W_{i,c}$ for $2\leq i\leq n$, and suppose that
$\Phi_c(\overline w_2,\ldots,\overline w_n)=1$. Then
\[
w_n\cdots w_2\in\gamma_{c+1}(D_n).
\]
Again by Proposition~\ref{prop-weighted-intersection}(2), there
exist $v_i\in W_{i,c+1}$, for $2\leq i\leq n$, such that
\[
w_n\cdots w_2=v_n\cdots v_2.
\] 
Since $w_i,v_i\in U_i$ 
for $2\leq i\leq n$, the uniqueness of the
normal form in Proposition~\ref{pro1} gives $w_i=v_i$ for every
$i=2,\ldots,n$. Hence $w_i\in W_{i,c+1}$ for every
$i=2,\ldots,n$, and therefore $\overline w_i=1$. Thus $\Phi_c$ is
injective.

Consequently,
\[
\frac{\gamma_c(D_n)}{\gamma_{c+1}(D_n)}
\cong
\bigoplus_{i=2}^{n}
\frac{W_{i,c}}{W_{i,c+1}}.
\]
By Lemma~\ref{weightedhall}(2), each summand is a finitely generated
free abelian group with basis given by the images of the basic
commutators of weighted degree $c$ in the free generators of $U_i$.
Thus the images of all these basic commutators, taken over
$i=2,\ldots,n$, form a finite basis of
$\gamma_c(D_n)/\gamma_{c+1}(D_n)$.
\end{proof}

\begin{corollary}\label{cor-residual}
For every $n\geq3$, the group $D_n$ is residually nilpotent.
\end{corollary}

\begin{proof}
Fix $n\geq3$, and let
$g\in\bigcap_{c\geq1}\gamma_c(D_n)$. By
Proposition~\ref{pro1}, write
\[
g=u_nu_{n-1}\cdots u_2,
\qquad
u_i\in U_i\quad(2\leq i\leq n).
\]
For every $c\geq1$,
Proposition~\ref{prop-weighted-intersection}(2) and the uniqueness
of the normal form give $u_i\in W_{i,c}$ for
$i=2,\ldots,n$. Hence, for $2\leq i\leq n$, 
$u_i\in\bigcap_{c\geq1}W_{i,c}=\{1\}$
by Lemma~\ref{weightedhall}(3). Thus $g=1$.
\end{proof}

\begin{proof}[Proof of Theorem~\ref{mainthm}]
\begin{enumerate}
\item For $i\geq3$, the equalities
\[
W_{i,c}
=
\gamma_c^{D_{i-1}}(U_i)
=
U_i\cap\gamma_c(D_i)
\]
are given by Proposition~\ref{prop-weighted-intersection}(1).
For $i=2$, the equality $W_{2,c}=U_2\cap\gamma_c(D_2)$
is immediate from $D_2=U_2\cong\mathbb Z$.
Finally, for $2\leq i\leq n$,
Proposition~\ref{prop-weighted-intersection}(2) gives
$U_i\cap\gamma_c(D_n)=W_{i,c}$.

\item This is Proposition~\ref{prop-weighted-intersection}(2).

\item The direct-sum decomposition and the stated basis are given by
Proposition~\ref{prop-lower-central-factors}. The assertion about
exact lower-central depth follows from Corollary~\ref{cor-exact-weight}.

\item By Corollary~\ref{cor-residual}, the group $D_n$ is residually
nilpotent. By Proposition~\ref{prop-lower-central-factors}, every
quotient $\gamma_c(D_n)/\gamma_{c+1}(D_n)$ is torsion-free. 
Therefore $D_n$ is a Magnus group.
\end{enumerate}
\end{proof}

\begin{remark}\label{rem1}
\upshape
Theorem~\ref{mainthm} implies, in particular, that $D_n$ is
residually torsion-free nilpotent. Indeed, let $1\neq g\in D_n$.
By residual nilpotence, there exists $c\geq1$ such that
$g\notin\gamma_{c+1}(D_n)$. Write $Q=D_n/\gamma_{c+1}(D_n)$.
Then the image of $g$ in $Q$ is non-trivial, and
$\gamma_{c+1}(Q)=\{1\}$, so $Q$ is nilpotent of class at most $c$.
It suffices to show that $Q$ is torsion-free.

For each $j=1,\ldots,c$, we have a natural isomorphism
\[
\frac{\gamma_j(Q)}{\gamma_{j+1}(Q)}
\cong
\frac{\gamma_j(D_n)}{\gamma_{j+1}(D_n)}.
\]
Hence $\gamma_j(Q)/\gamma_{j+1}(Q)$ is torsion-free by
Theorem~\ref{mainthm}.
Let $h$ be a torsion element of $Q$. We show inductively that
$h\in\gamma_j(Q)$ for every $j=1,\ldots,c+1$. This is clear for
$j=1$. If $h\in\gamma_j(Q)$, then its image in the torsion-free
group $\gamma_j(Q)/\gamma_{j+1}(Q)$
has finite order and is therefore trivial. Hence
$h\in\gamma_{j+1}(Q)$. It follows that
$h\in\gamma_{c+1}(Q)=\{1\}$.
Thus $Q$ is torsion-free.
Therefore the non-trivial image of $g$ is detected in the
torsion-free nilpotent quotient $Q$. Hence $D_n$ is residually
torsion-free nilpotent.

The residual torsion-free nilpotence of $D_n$ can alternatively be
obtained directly from \cite[Theorem~7.55]{darsu}, without using the
explicit identification of the relative lower central filtrations
obtained above. For $i=3,\ldots,n$,
Proposition~\ref{pro1} gives
$D_i=U_i\rtimes D_{i-1}$. The group $D_2=U_2$ is infinite cyclic,
and every $U_i$, for $i=2,\ldots,n$, is a finitely generated free
group. Hence $D_2$ and all the groups $U_i$ are residually
torsion-free nilpotent.

Let $3\leq i\leq n$ and assume inductively that $D_{i-1}$ is
residually torsion-free nilpotent. Write
\[
\varepsilon_r=\overline{d}_{i,r}\otimes1
\in U_i^{\mathrm{ab}}\otimes\mathbb Q,
\qquad
r=1,\ldots,i-1.
\]
For $1\leq l<j<i$, Lemma~\ref{lem1}(2) gives
\[
d_{j,l}^{-1}d_{i,j}d_{j,l}=d_{i,j}d_{i,l}^{-1}.
\]
Since $d_{j,l}$ commutes with $d_{i,l}$, it follows that
\[
d_{j,l}d_{i,j}d_{j,l}^{-1}=d_{i,j}d_{i,l}.
\]
Hence the conjugation action of $d_{j,l}$ on
$U_i^{\mathrm{ab}}\otimes\mathbb Q$, induced by
$u\mapsto d_{j,l}ud_{j,l}^{-1}$, sends
$\varepsilon_j$ to $\varepsilon_j+\varepsilon_l$ and fixes
$\varepsilon_r$ for $r\neq j$. Since the elements $d_{j,l}$, 
where $1\leq l<j<i$, generate
$D_{i-1}$, the action of $D_{i-1}$ preserves the filtration
\[
0=V_0<V_1<\cdots<V_{i-1}
=U_i^{\mathrm{ab}}\otimes\mathbb Q,
\]
where
\[
V_r=\operatorname{span}_{\mathbb Q}
\{\varepsilon_1,\ldots,\varepsilon_r\}
\qquad(1\leq r\leq i-1),
\]
and acts trivially on every quotient $V_r/V_{r-1}$. Thus the action
is unipotent. Applying \cite[Theorem~7.55]{darsu} with
$R=\mathbb Q$, we conclude that
$D_i=U_i\rtimes D_{i-1}$ is residually torsion-free nilpotent.
Induction on $i$ therefore shows that $D_n$ is residually
torsion-free nilpotent.
\end{remark}

\end{document}